\documentclass[11pt,draft,dvips]{amsart}  
\usepackage{amssymb,amsmath,amsfonts}  
\usepackage[dvips]{graphicx,color}
\usepackage{cite}
\allowdisplaybreaks[1]
\definecolor{pgray}{gray}{0.8}

\newtheorem{theorem}{Theorem}[section]

\newtheorem{proposition}[theorem]{Proposition}
\newtheorem{lemma}[theorem]{Lemma}

\numberwithin{equation}{section}
\title{\mbox{}}

\begin{document}
\begin{center}
{\bf \Large{
	Spatial-Decay of Solutions to the Quasi-Geostrophic Equation with the Critical and the Super-Critical Dissipation
}}\\
\vspace{5mm}
\vspace{5mm}
{\sc\large
	Masakazu Yamamoto}\footnote{Graduate School of Science and Technology,
	Niigata University,
	Niigata 950-2181, Japan}\qquad
{\sc\large
	Yuusuke Sugiyama}\footnote{Department of Engineering,
	The University of Shiga Prefecture,
	Hikone 522-8533, Japan}
\end{center}
\maketitle
\vspace{-15mm}
\begin{abstract}
The initial value problem for the two dimensional dissipative quasi-geostrophic equation derived from geophisical fluid dynamics is studied.
The dissipation of this equation is given by the fractional Laplacian.
It is known that the half Laplacian is a critical dissipation for the quasi-geostrophic equation.
In this paper, far field asymptotics of solutions are given in the critical and the supercritical cases.
\end{abstract}
%

\section{Introduction}
The quasi-geostrophic equation is derived from the model of geophisical fluid dynamics (see\cite{Cnstntn-Mjd-Tbk}).
Here we consider the following initial value problem:
\begin{equation}\label{qg}
\left\{
\begin{array}{lr}
	\partial_t \theta + (-\Delta)^{\alpha/2} \theta + \nabla \cdot (\theta\nabla^\bot\psi) = 0,
	&
	t > 0,~ x \in \mathbb{R}^2,\\
	(-\Delta)^{1/2}\psi = \theta,
	&
	t > 0,~ x \in \mathbb{R}^2,\\
	\theta (0,x) = \theta_0 (x),
	&
	x \in \mathbb{R}^2,
\end{array}
\right.
\end{equation}
where $0 < \alpha \le 2,~ \nabla^\bot = (-\partial/\partial x_2,\partial/\partial x_1)$ and the initial temperature $\theta_0$ is given as a nonnegative function.
The real valued function $\theta$ denotes the temperature and $\nabla^\bot \psi$ is the velocity.
The quasi-geostrophic equation is important also in the meteorology.
Since the Riesz transforms are in the nonlinear term and the fractional Laplacian leads an anomalous diffusion, the quesi-geostrophic equaion often is associated with the Navier-Stokes flow in the fluid mechanics.
When $1 < \alpha \le 2$, \eqref{qg} is a parabolic-type equation.
Then the smoothing effect guarantees well-posedness and regularity of global solutions in time.
This case is called the subcritical.
On the other hand, the case $\alpha = 1$ and the case $0 < \alpha < 1$ are the critical and the supercritical, respectively.
Even in those cases, existence and uniqueness of regular solutions in scale invariant spaces are discussed.
However some smoothness or smallness for the initial-data is required in order to show well-posedness of global solutions (cf.\cite{Cffrll-Vssur,Ch-L,C-C-W,Crdb-Crdb,J04,J05,K-N-V,Mur06}).
We treat the global regular solution of \eqref{qg} which satisfies
\begin{equation}\label{exist}
	\theta \in C \bigl( [0,\infty), L^1 (\mathbb{R}^2) \cap L^\infty (\mathbb{R}^2) \bigr),
	\quad
	\theta (t,x) \ge 0,
\end{equation}
and
\begin{equation}\label{mass}
	\int_{\mathbb{R}^2} \theta (t,x) dx
	=
	\int_{\mathbb{R}^2} \theta_0 (x) dx
	\quad\text{and}\quad
	\left\| \theta (t) \right\|_{L^\infty (\mathbb{R}^2)}
	\le
	C (1+t)^{-2/\alpha}
\end{equation}
for $t > 0$.
Those properties are confirmed for a smooth and small initial-data.
Moreover, if the initial-data is in $H^\sigma (\mathbb{R}^2)$ for some $\sigma > 2$ and sufficiently small, then
\begin{equation}\label{drbdd}
	\bigl\| (-\Delta)^{\sigma/2} \theta (t) \bigr\|_{L^2 (\mathbb{R}^2)}
	\le
	C (1+t)^{-\frac1\alpha-\frac\sigma\alpha}
\end{equation}
holds for $t > 0$ (see Proposition \ref{tr} in Section 2).
Asymptotic behavior of solutions of an equation of this type as $t \to + \infty$ is discussed in several preceding works (cf. for example \cite{Bl-Dl,Cr,E-Z,F-M,H-K-N,I-K-M,I,KtM_B,K-K,N-Y}).
In this paper we study spatial decay of the solution of \eqref{qg} by employing the following method.
For an unknown function $\varphi$, and a given and bounded function $\Phi$, now we assume that $\| |x|^\mu (\varphi-\Phi) \|_{L^2 (\mathbb{R}^2)} < + \infty$ and $\| |x|^\mu \Phi \|_{L^2 (\mathbb{R}^2)} = + \infty$, where $\mu$ is some positive constant.
Then $\Phi$ draws the spatial-decay of $\varphi$.
This idea firstly is applied to the Navier-Stokes flow and an asymptotic profile of the velocity as $|x| \to + \infty$ is derived (see \cite{Brndls,Brndls-Vgnrn}).
The solution of the quasi-geostrophic equation of subcritical case is estimated by the general theory via \cite{Brndls-Krch} which is developed for parabolic-type equations.
One fulfills that, if $1 < \alpha < 2$, then $\| |x|^\mu (\theta (t) - M G_\alpha (t)) \|_{L^2 (\mathbb{R}^2)} \le C_t$ and $\| |x|^\mu G_\alpha (t) \|_{L^2 (\mathbb{R}^2)} = + \infty$ for $1+\alpha \le \mu < 2 + \alpha$ and $t > 0$, where $M = \int_{\mathbb{R}^2} \theta_0 (x) dx$ and $G_\alpha (t) = \mathcal{F}^{-1} [e^{-t |\xi|^\alpha}]$ is the fundamental solution of $\partial_t \theta + (-\Delta)^{\alpha/2} \theta = 0$.
However this theory is not available for $0 < \alpha \le 1$ since \eqref{qg} is not a parabolic-type in this case.
The theory via \cite{Brndls-Krch} is based on the $L^p$-$L^q$ estimate for a mild solution.
The mild solution of \eqref{qg} is given by the following integral equation:
\[
	\theta (t)
	=
	G_\alpha (t) * \theta_0
	-
	\int_0^t
		\nabla G_\alpha (t-s) * (\theta \nabla^\bot (-\Delta)^{-1/2} \theta) (s)
	ds.
\]
In the subcritical case, $\nabla G_\alpha (t-s)$ in the nonlinear term is integrable in $s \in (0,t)$.
Therefore the $L^p$-$L^q$ estimate for $\nabla G_\alpha$ leads the assertion.
But, in the case $0 < \alpha \le 1$, this term has a singularity.
The goal of this paper is to derive spatial-decay of the solution of \eqref{qg} for $0 < \alpha \le 1$.
The main assertion is published as follows.
\begin{theorem}\label{sun}
Let $0 < \alpha \le 1,~ \sigma > 2,~ q > 2/\alpha,~ \theta_0 \in H^\sigma (\mathbb{R}^2),~ |x| \theta_0 \in L^1 (\mathbb{R}^2)$ and $|x|^2 \theta_0 \in L^2 (\mathbb{R}^2) \cap L^q (\mathbb{R}^2)$.
Assume that the solution $\theta$ of \eqref{qg} fulfills \eqref{exist}, \eqref{mass} and \eqref{drbdd}.
Then
\[
	\bigl\| |x|^2 \left( \theta (t) - M G_\alpha (t) \right) \bigr\|_{L^2 (\mathbb{R}^2)}
	\le
	\left\{
	\begin{array}{lr}
		C \left( \log (2+t) \right)^{3/2},
		&
		\alpha = 1,\\
		C \left( \log (2+t) \right)^{1/2},
		&
		0 < \alpha < 1\,
	\end{array}
	\right.
\]
holds for $t > 0$, where $M = \int_{\mathbb{R}^2} \theta_0 (x) dx$.
\end{theorem}
Here $M \in \mathbb{R}$ since $\| \theta_0 \|_{L^1 (\mathbb{R}^2)} \le C \| (1+|x|^2) \theta_0 \|_{L^2 (\mathbb{R}^2)} < + \infty$.
We emphasize that $\| |x|^2 G_\alpha (t) \|_{L^2 (\mathbb{R}^2)} = + \infty$ for $t>0$ (see \cite{Blmntr-Gtr}).
Therefore Theorem \ref{sun} states that the decay-rate of $\theta$ as $|x| \to +\infty$ is represented by $M G_\alpha$.
The logarithmic grows may not be crutial since the solution of the linear problem yields that $\| |x|^2 (G_\alpha (t) * \theta_0 - M G_\alpha (t)) \|_{L^2 (\mathbb{R}^2)} \le C$ for $t > 0$ (see Lemma \ref{wtasymplin} in Section 2).\\

\noindent
{\it Notation.}
We define the Fourier transform and its inverse by $\mathcal{F} [\varphi] (\xi) = (2\pi)^{-1} \int_{\mathbb{R}^2} e^{-ix\cdot\xi} \varphi (x) dx$ and $\mathcal{F}^{-1} [\varphi] (x) = (2\pi)^{-1} \int_{\mathbb{R}^2} e^{ix\cdot\xi} \varphi (\xi) d\xi$, where $i = \sqrt{-1}$.
We denote the derivations by $\partial_t = \partial / \partial t,~ \partial_j = \partial / \partial x_j~ (j = 1,2),~ \nabla = (\partial_1,\partial_2),~ \nabla^\bot = (-\partial_2,\partial_1),~ \Delta = \partial_1^2 + \partial_2^2$ and $(-\Delta)^{\alpha/2} \varphi = \mathcal{F}^{-1} [|\xi|^\alpha \mathcal{F} [\varphi]]$.
Also we define $(-\Delta)^{-\sigma/2} \varphi = \mathcal{F}^{-1} [|\xi|^{-\sigma} \mathcal{F} [\varphi]]$ for $0 < \sigma < 2$.
The H\"older conjugate of $1 \le p \le \infty$ is denoted by $p'$, i.e., $\frac1p + \frac1{p'} = 1$.
The Riesz transform is defined by $R_j \varphi = \partial_j (-\Delta)^{-1/2} \varphi = \mathcal{F}^{-1} [i\xi_j |\xi|^{-1} \mathcal{F} [\varphi]]~ (j = 1,2)$.
For $\beta = (\beta_1,\beta_2) \in \mathbb{Z}_+^2 = (\mathbb{N} \cup \{ 0 \})^2,~ |\beta| = \beta_1 + \beta_2$.
For some operators $A$ and $B$, we denote the commutator by $[A,B] = AB - BA$.
Various nonnegative constants are denoted by $C$.

\section{Preliminaries}
In this section we prepare some inequalities for several functions and the solution.
From the scaling property, the fundamental solution fulfills that
\begin{equation}\label{scG}
	G_\alpha (t,x) = t^{-2/\alpha} G_\alpha (1,t^{-1/\alpha} x)
\end{equation}
for $(t,x) \in (0,+\infty) \times \mathbb{R}^2$.
Furthermore
\begin{equation}\label{decayG}
	\bigl| \nabla^\beta G_\alpha (1,x) \bigr|
	\le
	C_\beta (1+|x|^2)^{-1-\frac\alpha{2}-\frac{|\beta|}2}
\end{equation}
is satisfied for $\beta \in \mathbb{Z}_+^2$ and $x \in \mathbb{R}^2$.
When $\alpha = 1$, this estimate is clear since $G_\alpha$ is the Poisson kernel in this case.
For the case $0 < \alpha < 1$, we use the following H\"ormander-Mikhlin inequality.
\begin{lemma}[H\"ormander-Mikhlin inequality\cite{Hrmndr,Mkhln}]\label{lem-HM}
Let $N \in \mathbb{Z}_+,~ 0 < \mu < 1$ and $\lambda = N + \mu - 2$.
Assume that $\varphi \in C^\infty (\mathbb{R}^2\backslash \{ 0 \})$ satisfies the following conditions:
\begin{itemize}
\item
	$\nabla^\gamma \varphi \in L^1 (\mathbb{R}^2)$ for any $\gamma \in \mathbb{Z}_+^2$ with $|\gamma| \le N$;
\item
	$|\nabla^\gamma \varphi (\xi)| \le C_\gamma |\xi|^{\lambda - |\gamma|}$ for $\xi \neq 0$ and $\gamma \in \mathbb{Z}_+^2$ with $|\gamma| \le N+1$.
\end{itemize}
Then
\[
	\sup_{x \neq 0} \bigl( |x|^{2+\lambda} \bigl| \mathcal{F}^{-1} [\varphi] (x) \bigr| \bigr)
	< + \infty
\]
holds.
\end{lemma}
The proof and the details of this lemma is in \cite{S-S}.
We confirm \eqref{decayG} when $0 < \alpha < 1$.
If $|\beta| = 2k$ for $k \in \mathbb{Z}_+$, then we put
$
	\varphi (\xi)
	=
	\nabla (-\Delta)^k ( \xi^\beta e^{-|\xi|^\alpha} ).
$
Then $\varphi$ satisfies the conditions in Lemma \ref{lem-HM} with $N = 1,~ \mu = \alpha$ and $\lambda = \alpha - 1$.
Hence, since $\nabla^\beta G_\alpha (1,x)$ is bounded, we see \eqref{decayG}.
When $|\beta| = 2k+1$, we put $\varphi (\xi) = (-\Delta)^{k+1} (\xi^\beta e^{-|\xi|^\alpha})$ in the above procedure and derive \eqref{decayG}.
The following relation plays important role in the energy method.
\begin{lemma}[Stroock-Varopoulos inequality \cite{L-S}]\label{lemC-C}
	Let $0 \le \alpha \le 2,~ q \ge 2$ and $f \in W^{\alpha,q} (\mathbb{R}^2)$.
	Then
	\[
		\int_{\mathbb{R}^2}
			|f|^{q-2} f (-\Delta)^{\alpha/2} f
		dx
		\ge
		\frac2q \int_{\mathbb{R}^2}
			\left| (-\Delta)^{\alpha/4} (|f|^{q/2}) \right|^2
		dx
	\]
	holds.
\end{lemma}
For the proof of this lemma, see \cite{Crdb-Crdb,J05}.
The fractional integral $(-\Delta)^{-\sigma/2} \varphi$ for $0 < \sigma < 2$ is defined by $(-\Delta)^{-\sigma/2} \varphi = \mathcal{F}^{-1} [|\xi|^{-\sigma} \mathcal{F} [\varphi]]$ and represented by
\begin{equation}\label{fracint}
	(-\Delta)^{-\sigma/2} \varphi (x)
	=
	\gamma_\sigma \int_{\mathbb{R}^2}
		\frac{\varphi(y)}{|x-y|^{2-\sigma}}
	dy
\end{equation}
for some constant $\gamma_\sigma$ (see \cite{S,Z}).
For this integral we see the following inequality of Sobolev type.
\begin{lemma}[Hardy-Littlewood-Sobolev's inequality \cite{S,Z}]\label{HLS}~
	Let $0 < \sigma < 2,~ 1 < p < \frac{2}{\sigma}$ and $\frac1{p_*} = \frac1p - \frac{\sigma}2$.
	Then there exists a positive constant $C$ such that
	\[
		\bigl\| (-\Delta)^{-\sigma/2} \varphi \bigr\|_{L^{p_*} (\mathbb{R}^2)}
		\le
		C \bigl\| \varphi \bigr\|_{L^p (\mathbb{R}^2)}
	\]
	for any $\varphi \in L^p (\mathbb{R}^2)$.
\end{lemma}
We also need the following generalized Gagliardo-Nirenberg inequality.
\begin{lemma}[Gagliardo-Nirenberg inequality \cite{H-Y-Z,K-PP,M-N-S-S}]\label{GN}
	Let $0 < \sigma < s < 2,~ 1 < p_1, p_2 < \infty$ and $\frac1p = (1-\frac\sigma{s}) \frac1{p_1} + \frac\sigma{s} \frac1{p_2}$.
	Then
	\[
		\bigl\| (-\Delta)^{\sigma/2} \varphi \bigr\|_{L^p (\mathbb{R}^2)}
		\le
		C \bigl\| \varphi \bigr\|_{L^{p_1} (\mathbb{R}^2)}^{1-\frac\sigma{s}}
		\bigl\| (-\Delta)^{s/2} \varphi \bigr\|_{L^{p_2} (\mathbb{R}^2)}^{\frac\sigma{s}}
	\]
	holds.
\end{lemma}
The following estimate is due to \cite{K-P}.
\begin{lemma}[Kato-Ponce's commutator estimates \cite{J04,K-P}]\label{cmm}
	Let $s > 0$ and $1 < p < \infty$.
	Then
	\[
	\begin{split}
		\bigl\| [(-\Delta)^{s/2}, g] f \bigr\|_{L^p (\mathbb{R}^n)}
		\le&
		C \bigl( \| \nabla g \|_{L^{p_1} (\mathbb{R}^n)} \| (-\Delta)^{(s-1)/2} f \|_{L^{p_2} (\mathbb{R}^n)}
		+ \| (-\Delta)^{s/2} g \|_{L^{p_3} (\mathbb{R}^n)} \| f \|_{L^{p_4} (\mathbb{R}^n)} \bigr)
	\end{split}
	\]
	with $1 < p_j \le \infty~ (j = 1,4)$ and $1 < p_j < \infty~ (j = 2,3)$ such that $\frac1p = \frac1{p_1} + \frac1{p_2} = \frac1{p_3} + \frac1{p_4}$, where $[(-\Delta)^{s/2}, g] f = (-\Delta)^{s/2} (gf) - g (-\Delta)^{s/2} f$.
\end{lemma}
By using those inequalities, we firstly confirm \eqref{drbdd}.
\begin{proposition}\label{tr}
Let $\sigma > 2,~ \theta_0 \in H^\sigma (\mathbb{R}^2)$ and $\| \theta_0 \|_{H^\sigma (\mathbb{R}^2)}$ be small.
Assume that the solution $\theta$ of \eqref{qg} satisfies \eqref{exist} and \eqref{mass}.
Then \eqref{drbdd} holds.
\end{proposition}
\begin{proof}
We trace the proof of \cite[Theorem 3.1]{J04}.
Similar argument as in \cite{Y-S-NA} leads that
\begin{equation}\label{bstr}
\begin{split}
	&\frac12 (1+t)^{2\gamma} \bigl\| (-\Delta)^{\varsigma/2} \theta \bigr\|_{L^2 (\mathbb{R}^2)}^2
	+
	\int_0^t
		(1+s)^{2\gamma} \bigl\| (-\Delta)^{\frac\varsigma{2}+\frac\alpha{4}} \theta \bigr\|_{L^2 (\mathbb{R}^2)}^2
	ds\\
	&=
	\frac12  \bigl\| (-\Delta)^{\varsigma/2} \theta_0 \bigr\|_{L^2 (\mathbb{R}^2)}^2
	+
	\gamma \int_0^t (1+s)^{2\gamma-1} \bigl\| (-\Delta)^{\varsigma/2} \theta \bigr\|_{L^2 (\mathbb{R}^2)}^2 ds\\
	&-
	\int_0^t
		(1+s)^{2\gamma}
		\int_{\mathbb{R}^2}
			(-\Delta)^{\varsigma/2} \theta \bigl[ (-\Delta)^{\varsigma/2}, \nabla^\bot \psi \bigr] \cdot \nabla \theta
		dy
	ds
\end{split}
\end{equation}
for $0 < \varsigma \le \sigma$ and $\gamma > \frac1\alpha+\frac\sigma\alpha$.
For the last term, we apply Lemma \ref{cmm} with $p = 2_*'$ and $p_2 = p_3 = 2_*$, where $\frac1{2_*} = \frac12 - \frac\alpha{4}$, and $p_1 = p_4 = 2/\alpha$, and Lemma \ref{HLS}, then
\begin{equation}\label{bstr2}
\begin{split}
	&\biggl|
		\int_{\mathbb{R}^2}
			(-\Delta)^{\varsigma/2} \theta \bigl[ (-\Delta)^{\varsigma/2}, \nabla^\bot \psi \bigr] \cdot \nabla \theta
		dy
	\biggr|
	\le
	C \bigl\| \nabla\theta \bigr\|_{L^{2/\alpha} (\mathbb{R}^2)}
	\bigl\| (-\Delta)^{\varsigma/2} \theta \bigr\|_{L^{2_*} (\mathbb{R}^2)}^2\\
	&\le
	C \bigl\| (-\Delta)^{1-\frac\alpha{2}} \theta \bigr\|_{L^2 (\mathbb{R}^2)}
	\bigl\| (-\Delta)^{\frac\varsigma{2}+\frac\alpha{4}} \theta \bigr\|_{L^2 (\mathbb{R}^2)}^2.
\end{split}
\end{equation}
Here $\| (-\Delta)^{1-\frac\alpha{2}} \theta \|_{L^2 (\mathbb{R}^2)}$ is bounded by $\| \theta_0 \|_{H^\sigma (\mathbb{R}^2)}$ (cf. \cite{J04}) which is small.
By Lemma \ref{GN} and the Young inequality, we see for the second term of \eqref{bstr} that
\begin{equation}\label{bstr1}
\begin{split}
	&\gamma \int_0^t (1+s)^{2\gamma-1} \bigl\| (-\Delta)^{\varsigma/2} \theta \bigr\|_{L^2 (\mathbb{R}^2)}^2 ds\\
	&\le
	C_\delta \int_0^t
		(1+s)^{ 2\gamma - 1 - \frac{2\varsigma}\alpha} \bigl\| \theta \bigr\|_{L^2 (\mathbb{R}^2)}^2
	ds
	+
	\delta \int_0^t
		(1+s)^{2\gamma}  \bigl\| (-\Delta)^{\frac\varsigma{2}+\frac\alpha{4}} \theta \bigr\|_{L^2 (\mathbb{R}^2)}^2
	ds
\end{split}
\end{equation}
for any small $\delta > 0$.
We see from \eqref{mass} that the first term of this inequality is bounded by $(1+t)^{2\gamma-\frac2\alpha-\frac{2\varsigma}\alpha}$.
The estimate \eqref{bstr1} is guaranteed if $\varsigma < 2 - \alpha/2$.
Hence, by applying \eqref{bstr2} and \eqref{bstr1} into \eqref{bstr}, we see \eqref{drbdd} with $\varsigma < 2 - \alpha/2$ instead of $\sigma$.
Choosing $\varsigma < 4 - \alpha$, we have from Lemma \ref{GN} that
\begin{equation}\label{bstr1p}
\begin{split}
	&\gamma \int_0^t (1+s)^{2\gamma-1} \bigl\| (-\Delta)^{\varsigma/2} \theta \bigr\|_{L^2 (\mathbb{R}^2)}^2 ds\\
	&\le
	C_\delta \int_0^t
		(1+s)^{2\gamma - 1 - \frac{2\varsigma}\alpha + \frac{2\varsigma_1}\alpha} \bigl\| (-\Delta)^{\varsigma_1/2} \theta \bigr\|_{L^2 (\mathbb{R}^2)}^2
	ds
	+
	\delta \int_0^t
		(1+s)^{2\gamma}  \bigl\| (-\Delta)^{\frac\varsigma{2}+\frac\alpha{4}} \theta \bigr\|_{L^2 (\mathbb{R}^2)}^2
	ds
\end{split}
\end{equation}
for some $\varsigma_1 < 2 - \alpha/2$.
Here \eqref{drbdd} with $\varsigma_1$ instead of $\sigma$ yields that the first term of this inequality is bounded by $(1+t)^{2\gamma - \frac2\alpha-\frac{2\varsigma}\alpha}$.
Thus \eqref{bstr} together with \eqref{bstr2} and \eqref{bstr1p} gives \eqref{drbdd} with $\varsigma < 4 - \alpha$ instead of $\sigma$.
By repeating this procedure, we can choose $\varsigma = \sigma$ and conclude the proof.
\end{proof}
For this estimate, suitable conditions for the initial-data are discussed in \cite{J04}.
The solution of \eqref{qg} is included in the weighted Lebesgue spaces.
\begin{proposition}\label{prop-wt}
Let $q > 2/\alpha$ and $|x|^2 \theta_0 \in L^q (\mathbb{R}^2)$.
Assume that the solution $\theta$ of \eqref{qg} satiefies \eqref{exist} and \eqref{mass}.
Then
\[
	\bigl\| |x|^2 \theta (t) \bigr\|_{L^q (\mathbb{R}^2)}
	\le
	C (1+t)^{\frac{2}{\alpha q}}
\]
for $t > 0$.
\end{proposition}
\begin{proof}
We put $\Theta (t,x) = |x|^2 \theta (t,x)$, then we see
\[
	\partial_t \Theta + (-\Delta)^{\alpha/2} \Theta + \nabla \cdot (\Theta \nabla^\bot \psi)
	=
	\bigl[ (-\Delta)^{\alpha/2}, |x|^2 \bigr] \theta
	-
	\bigl[ |x|^2, \nabla \bigr] \cdot (\theta \nabla^\bot \psi).
\]
Here
\begin{equation}\label{str}
\begin{split}
	\bigl[ (-\Delta)^{\alpha/2}, |x|^2 \bigr] \theta
	&=
	\mathcal{F}^{-1} \bigl[ \bigl[ |\xi|^\alpha, -\Delta \bigr] \hat{\theta} \bigr]
	=
	\mathcal{F}^{-1} \bigl[ \alpha (\alpha-1) |\xi|^{\alpha-2} \hat{\theta} + 2\alpha |\xi|^{\alpha-2} \xi \cdot \nabla \hat{\theta} \bigr]\\
	&=
	\alpha (\alpha-1) (-\Delta)^{(\alpha-2)/2} \theta - 2\alpha (-\Delta)^{(\alpha -2)/2} \nabla \cdot (x \theta)
\end{split}
\end{equation}
and
\[
\begin{split}
	\bigl[ |x|^2, \nabla \bigr] \cdot (\theta \nabla^\bot \psi)
	&=
	\mathcal{F}^{-1} \bigl[ [ -\Delta,i\xi] \cdot \mathcal{F} [\theta \nabla^\bot \psi] \bigr]
	=
	-2 \mathcal{F}^{-1} \bigl[ i\nabla \cdot \mathcal{F} [\theta \nabla^\bot \psi] \bigr]
	=
	-2 x \cdot (\theta \nabla^\bot \psi).
\end{split}
\]
Hence
\[
\begin{split}
	&\frac1{q} \frac{d}{dt} \| \Theta (t) \|_{L^q (\mathbb{R}^2)}^q
	+
	\frac2{q} \| (-\Delta)^{\alpha/4} (\Theta^{q/2}) \|_{L^2 (\mathbb{R}^2)}^2
	\le
	\alpha (\alpha - 1) \int_{\mathbb{R}^2} \Theta^{q-1} (-\Delta)^{(\alpha-2)/2} \theta dx\\
	&-
	2\alpha \int_{\mathbb{R}^2} \Theta^{q-1} (-\Delta)^{(\alpha-2)/2} \nabla \cdot (x\theta) dx
	+
	2 \int_{\mathbb{R}^2} \Theta^{q-1} x \cdot (\theta \nabla^\bot \psi) dx.
\end{split}
\]
Here we used Lemma \ref{lemC-C} and the fact that $\int_{\mathbb{R}^2} \Theta^{q-1} \nabla \cdot (\Theta\nabla^\bot \psi) dx = 0$.
For $0 < \gamma < \frac2{\alpha q}$, we multiply this inequality by $(1+t)^{-\gamma q}$, integrate over $(0,t)$ and give that
\begin{equation}\label{bswt}
\begin{split}
	&(1+t)^{-\gamma q} \| \Theta (t) \|_{L^q (\mathbb{R}^2)}^q
	+
	\gamma q^2 \int_0^t (1+s)^{-\gamma q - 1} \| \Theta(s) \|_{L^q (\mathbb{R}^2)}^q ds\\
	&+
	2 \int_0^t (1+s)^{-\gamma q} \| (-\Delta)^{\alpha/4} (\Theta^{q/2}) \|_{L^2 (\mathbb{R}^2)}^2 ds\\
	&\le
	\| |x|^2 \theta_0 \|_{L^q (\mathbb{R}^2)}^q
	+
	\alpha (\alpha - 1) q \int_0^t (1+s)^{-\gamma q} \int_{\mathbb{R}^2} \Theta^{q-1} (-\Delta)^{(\alpha-2)/2} \theta dxds\\
	&-
	2\alpha q \int_0^t (1+s)^{-\gamma q} \int_{\mathbb{R}^2} \Theta^{q-1} (-\Delta)^{(\alpha-2)/2} \nabla \cdot (x\theta) dxds
	+
	2q \int_0^t (1+s)^{-\gamma q} \int_{\mathbb{R}^2} \Theta^{q-1} x \cdot (\theta \nabla^\bot \psi) dxds.
\end{split}
\end{equation}
Hardy-Littlewood-Sobolev's inequality with $\sigma = 2-\alpha$ provides that
\[
\begin{split}
	&\int_{\mathbb{R}^2} \Theta^{q-1} (-\Delta)^{(\alpha-2)/2} \theta dx
	\le
	\| \Theta \|_{L^q (\mathbb{R}^2)}^{q-1} \| (-\Delta)^{(\alpha-2)/2} \theta \|_{L^q (\mathbb{R}^2)}
	\le
	C \| \theta \|_{L^{r} (\mathbb{R}^2)} \| \Theta \|_{L^q (\mathbb{R}^2)}^{q-1}\\
	&\le
	C (1+t)^{-1+\frac2{\alpha q}} \| \Theta \|_{L^q (\mathbb{R}^2)}^{q-1},
\end{split}
\]
where $\frac1r = \frac1q + \frac{2-\alpha}2 $.
Thus
\[
\begin{split}
	&\int_0^t (1+s)^{-\gamma q} \int_{\mathbb{R}^2} \Theta^{q-1} (-\Delta)^{(\alpha-2)/2} \theta dxds
	\le
	C \int_0^t
		(1+s)^{-\gamma q -1 + \frac2{\alpha q}} \| \Theta \|_{L^q (\mathbb{R}^2)}^{q-1}
	ds\\
	&\le
	C_\delta \int_0^t
		(1+s)^{-\gamma q - 1 + \frac2\alpha}
	ds
	+
	\delta \int_0^t
		(1+s)^{-\gamma q - 1} \| \Theta \|_{L^q (\mathbb{R}^2)}^q
	ds
\end{split}
\]
for $\delta > 0$.
Similarly, for $\frac1r = \frac1q + \frac12 - \frac{\alpha}2$,
\[
\begin{split}
	&\int_{\mathbb{R}^2} \Theta^{q-1} (-\Delta)^{(\alpha-2)/2} \nabla \cdot (x\theta) dx
	\le
	\| \Theta \|_{L^q (\mathbb{R}^2)}^{q-1} \| (-\Delta)^{(\alpha -2)/2} \nabla \cdot (x\theta) \|_{L^q (\mathbb{R}^2)}
	\le
	C \| x\theta \|_{L^r (\mathbb{R}^2)} \| \Theta \|_{L^q (\mathbb{R}^2)}^{q-1}\\
	&\le
	C \| \theta \|_{L^{rq/(2q-r)} (\mathbb{R}^2)}^{1/2} \| \Theta \|_{L^q (\mathbb{R}^2)}^{q-\frac12}
	\le
	C (1+t)^{-1 + \frac1{\alpha q}} \| \Theta \|_{L^q (\mathbb{R}^2)}^{q-\frac12}.
\end{split}
\]
Therefore
\[
\begin{split}
	&\int_0^t (1+s)^{-\gamma q} \int_{\mathbb{R}^2} \Theta^{q-1} (-\Delta)^{(\alpha-2)/2} \nabla \cdot (x\theta) dxds
	\le
	C \int_0^t (1+s)^{-\gamma q - 1 + \frac1{\alpha q}} \| \Theta \|_{L^q (\mathbb{R}^2)}^{q-\frac12} ds\\
	&\le
	C_\delta \int_0^t
		(1+s)^{-\gamma q - 1 + \frac2\alpha}
	ds
	+
	\delta \int_0^t
		(1+s)^{-\gamma q - 1} \| \Theta \|_{L^q (\mathbb{R}^2)}^q
	ds.
\end{split}
\]
Since $\| \nabla^\bot \psi \|_{L^r (\mathbb{R}^2)} = \| (-R_2 \theta,R_1 \theta) \|_{L^r (\mathbb{R}^2)} \le C \| \theta \|_{L^r (\mathbb{R}^2)}$ for $1 < r < \infty$,
\[
\begin{split}
	&\int_{\mathbb{R}^2} \Theta^{q-1} x \cdot (\theta \nabla^\bot \psi) dxds
	\le
	\| \Theta \|_{L^q (\mathbb{R}^2)}^{q-1} \| x \theta \|_{L^{2q} (\mathbb{R}^2)} \| \nabla^\bot \psi \|_{L^{2q} (\mathbb{R}^2)}
	\le
	C \| \theta \|_{L^\infty (\mathbb{R}^2)}^{1/2} \| \theta \|_{L^{2q} (\mathbb{R}^2)} \| \Theta \|_{L^q (\mathbb{R}^2)}^{q-\frac12}\\
	&\le
	C (1+t)^{-\frac3\alpha + \frac1{\alpha q}} \| \Theta \|_{L^q (\mathbb{R}^2)}^{q-\frac12}.
\end{split}
\]
Hence
\[
\begin{split}
	&\int_0^t (1+s)^{-\gamma q} \int_{\mathbb{R}^2} \Theta^{q-1} x \cdot (\theta \nabla^\bot \psi) dxds
	\le
	C \int_0^t
		(1+s)^{-\gamma q - \frac3\alpha + \frac1{\alpha q}} \| \Theta \|_{L^q (\mathbb{R}^2)}^{q-\frac12}
	ds\\
	&\le
	C_\delta \int_0^t
		(1+s)^{-\gamma q - 1 + \frac2\alpha - \frac{6q}\alpha + 2q}
	ds
	+
	\delta \int_0^t
		(1+s)^{-\gamma q - 1} \| \Theta \|_{L^q (\mathbb{R}^2)}^q
	ds.
\end{split}
\]
By applying those inequalities into \eqref{bswt} and choosing $\delta$ sufficiently small, we conclude the proof.
\end{proof}
This proposition is crucial since $\| |x|^2 G_\alpha (t) \|_{L^q (\mathbb{R}^2)} = t^{\frac2{\alpha q}} \| |x|^2 G_\alpha (1) \|_{L^q (\mathbb{R}^2)}$ and $\| |x|^2 G_\alpha (1) \|_{L^q (\mathbb{R}^2)} < +\infty$.
We see the spatial-decay of the solution of the linear equation as in follows.

\begin{lemma}\label{wtasymplin}
Let $0 < \alpha \le 2$, then
\[
	\bigl\| |x|^2 \left( G_\alpha (t) * \theta_0 - M G_\alpha (t) \right) \bigr\|_{L^2 (\mathbb{R}^2)}
	\le C_\alpha \bigl( \bigl\| |x| \theta_0 \bigr\|_{L^1 (\mathbb{R}^2)} + \bigl\| |x|^2 \theta_0 \bigr\|_{L^2 (\mathbb{R}^2)} \bigr)
\]
holds for $\theta_0 \in L^1 (\mathbb{R}^2)$ with $|x| \theta_0 \in L^1 (\mathbb{R}^2)$ and $|x|^2 \theta_0 \in L^2 (\mathbb{R}^2)$, where $M = \int_{\mathbb{R}^2} \theta_0 (x) dx$.
\end{lemma}
\begin{proof}
The mean-value theorem yields that
\[
\begin{split}
	G_\alpha (t) * \theta_0 - M G_\alpha (t)
	&=
	\int_{|y| \le |x|/2} \int_0^1
		\nabla G_\alpha (t,x-\lambda y) \cdot (-y) \theta_0 (y)
	d\lambda dy\\
	&+ \int_{|y| > |x|/2}
		\left( G_\alpha (t,x-y) - G_\alpha (t,x) \right) \theta_0 (y)
	dy.
\end{split}
\]
Hence Hausdorff-Young's inequality with \eqref{scG} and \eqref{decayG} gives that
\[
\begin{split}
	&\left\| |x|^2 \left( G_\alpha (t) * \theta_0 - M G_\alpha (t) \right) \right\|_{L^2 (\mathbb{R}^2)}\\
	&\le
	C \left\| |x|^2 \nabla G_\alpha (t) \right\|_{L^2 (\mathbb{R}^2)} \bigl\| |x| \theta_0 \bigr\|_{L^1 (\mathbb{R}^2)}
	+
	C \bigl\| G_\alpha (t) \bigr\|_{L^1 (\mathbb{R}^2)} \bigl\| |x|^2 \theta_0 \bigr\|_{L^2 (\mathbb{R}^2)}
	+
	C \bigl\| |x| G_\alpha (t) \bigr\|_{L^2 (\mathbb{R}^2)} \bigl\| |x| \theta_0 \bigr\|_{L^1 (\mathbb{R}^2)}\\
	&\le
	C
\end{split}
\]
and we coclude the proof.
\end{proof}

\section{Proof of main theorem}
We prove Theorem \ref{sun}.
Put $v = \theta - G_\alpha * \theta_0$, then
\[
	v = -\int_0^t \nabla G_\alpha (t-s)* (\theta \nabla^\bot \psi) (s) ds.
\]
For $1 \le p < \frac2{1-\alpha}$, we choose $r_1$ and $r_2$ such that $1 + \frac1{p} = \frac1{r_1} + \frac1{r_2},~ \frac2\alpha (1-\frac1{r_1}) < 1$ and $1	 \le r_2 < 2$. Moreover, let $\varsigma \le \sigma-1$, then we see from Hausdorff-Young's inequality and Proposition \ref{tr} that
\[
\begin{split}
	&\| (-\Delta)^{\varsigma/2} v \|_{L^p (\mathbb{R}^2)}
	\le
	\int_0^{t/2}
		\bigl\| \nabla (-\Delta)^{\varsigma/2} G_\alpha (t-s) \bigr\|_{L^p (\mathbb{R}^2)}
		\bigl\| \theta \nabla^\bot\psi \bigr\|_{L^1 (\mathbb{R}^2)}
	ds\\
	&+
	\int_{t/2}^t
		\bigl\| G_\alpha (t-s) \bigr\|_{L^{r_1} (\mathbb{R}^2)}
		\bigl\| \nabla (-\Delta)^{\varsigma/2} \cdot (\theta \nabla^\bot \psi ) \bigr\|_{L^{r_2} (\mathbb{R}^2)}
	ds\\
	&\le
	C \int_0^{t/2}
		(t-s)^{-\frac2\alpha (1-\frac1{p}) - \frac1\alpha -\frac{\varsigma}\alpha} (1+s)^{-\frac2\alpha}
	ds
	+
	C \int_{t/2}^t
		(t-s)^{-\frac2\alpha (1-\frac1{r_1})}
		(1+s)^{-\frac2\alpha (1-\frac1{r_2})-\frac3\alpha -\frac\varsigma\alpha}
	ds.
\end{split}
\]
Similarly
\[
	\| (-\Delta)^{\varsigma/2} v \|_{L^p (\mathbb{R}^2)}
	\le
	C \int_0^t
		(t-s)^{-\frac2\alpha (1-\frac1{r_1})}
		(1+s)^{-\frac2\alpha (1-\frac1{r_2})-\frac3\alpha -\frac\varsigma\alpha}
	ds.
\]
Thus
\begin{equation}\label{decay-drv}
	\| (-\Delta)^{\varsigma/2} v \|_{L^p (\mathbb{R}^2)}
	\le
	C (1+t)^{-\frac2\alpha (1-\frac1{p}) - \frac1\alpha -\frac\varsigma\alpha}
\end{equation}
for $1 \le p < \frac2{1-\alpha}$ and $\varsigma \le \sigma - 1$.
It also holds that
\begin{equation}\label{v}
\left\{
\begin{array}{lr}
	\partial_t v + (-\Delta)^{\alpha/2} v + \nabla \cdot (\theta\nabla^\bot\psi) = 0,
	&
	t > 0,~ x \in \mathbb{R}^2,\\
	v (0,x) = 0,
	&
	x \in \mathbb{R}^2.
\end{array}
\right.
\end{equation}
Thus
\begin{equation}\label{bs}
\begin{split}
	&\frac12 \bigl\| |x|^2 v(t) \bigr\|_{L^2 (\mathbb{R}^2)}^2
	+
	\int_0^t
		\bigl\| (-\Delta)^{\alpha/4} (|x|^2 v) \bigr\|_{L^2 (\mathbb{R}^2)}^2
	ds
	=
	4 \int_0^t \int_{\mathbb{R}^2}
		|x|^{2} \theta v x \cdot \nabla^\bot \psi
	dxds\\
	&+
	\int_0^t \int_{\mathbb{R}^2}
		|x|^{4} \theta \nabla v \cdot \nabla^\bot \psi
	dxds
	- \int_0^t \int_{\mathbb{R}^2}
		|x|^2 v \Bigl[ |x|^2, (-\Delta)^{\alpha/2} \Bigr] v
	dxds.
\end{split}
\end{equation}
We see from \eqref{mass} and Proposition \ref{prop-wt} that
\[
\begin{split}
	&\biggl| \int_{\mathbb{R}^2}
		|x|^2 \theta v x \cdot \nabla^\bot \psi
	dx \biggr|
	\le
	C \bigl\| |x| \theta \bigr\|_{L^{2q} (\mathbb{R}^2)} 
	\bigl\| |x|^2 v \bigr\|_{L^{2_*} (\mathbb{R}^2)}
	\bigl\| \nabla^\bot \psi \bigr\|_{L^r (\mathbb{R}^2)}\\
	&\le
	C \bigl\| \theta \bigr\|_{L^\infty (\mathbb{R}^2)}^{1/2}
	\bigl\| |x|^2 \theta \bigr\|_{L^q (\mathbb{R}^2)}^{1/2}
	 \bigl\| \theta \bigr\|_{L^r (\mathbb{R}^2)}
	\bigl\| (-\Delta)^{\alpha/4} (|x|^2 v) \bigr\|_{L^2 (\mathbb{R}^2)}
	\le
	C (1+t)^{- \frac2\alpha + \frac12} \bigl\| (-\Delta)^{\alpha/4} (|x|^2 v) \bigr\|_{L^2 (\mathbb{R}^2)},
\end{split}
\]
where $\frac1{2_*} = \frac12 - \frac\alpha{4}$ and $\frac1r = \frac12 + \frac\alpha{4} - \frac1{2q}$.
Hence
\begin{equation}\label{bs1}
	\biggl| \int_0^t \int_{\mathbb{R}^2}
		|x|^{2} \theta v x \cdot \nabla^\bot \psi
	dxds \biggr|
	\le
	C_\delta
	+
	\delta \int_0^t
		\bigl\| (-\Delta)^{\alpha/4} (|x|^2 v) \bigr\|_{L^2 (\mathbb{R}^2)}^2
	ds
\end{equation}
for $\delta > 0$.
For the second term of \eqref{bs}, we have
\[
\begin{split}
	\biggl| \int_{\mathbb{R}^2}
		|x|^4 \theta \nabla v \cdot \nabla^\bot \psi
	dx \biggr|
	\le
	C \bigl\| |x|^{2} \theta \bigr\|_{L^q (\mathbb{R}^2)}
	\bigl\| |x|^2 \nabla v \cdot \nabla^\bot \psi \bigr\|_{L^{q'} (\mathbb{R}^2)}.
\end{split}
\]
Now $\nabla^\bot \psi = (-R_2\theta, R_1\theta)$ and
\[
\begin{split}
	|x|^2 R_j \varphi
	&=
	\gamma \int_{\mathbb{R}^2}
		\left( \frac{x_j-y_j}{|x-y|} + \frac{2(x-y)\cdot y (x_j-y_j)}{|x-y|^3} + \frac{|y|^2 (x_j-y_j)}{|x-y|^3} \right)
		\varphi (y)
	dy
\end{split}
\]
for any suitable function $\varphi$, where $\gamma = \pi^{-3/2} \Gamma (3/2)$.
Furthermore
\begin{gather*}
	\biggl| \int_{\mathbb{R}^2}
		\frac{x_j-y_j}{|x-y|} \varphi (y)
	dy \biggr|
	\le
	\| \varphi \|_{L^1 (\mathbb{R}^2)},\\
	\biggl\| \int_{\mathbb{R}^2}
		\frac{(x-y)\cdot y (x_j-y_j)}{|x-y|^3} \varphi (y)
	dy \biggr\|_{L^{2q} (\mathbb{R}^2)}
	\le
	C \bigl\| (-\Delta)^{-1/2} (|x| \varphi) \bigr\|_{L^{2q} (\mathbb{R}^2)}
	\le
	C \bigl\| |x| \varphi \bigr\|_{L^r (\mathbb{R}^2)}
\end{gather*}
for $\frac1{r} = \frac12 + \frac1{2q}$, and
\[
	\gamma \int_{\mathbb{R}^2}
		\frac{|y|^2 (x_j-y_j)}{|x-y|^3} \varphi (y)
	dy
	= R_j (|x|^2\varphi).
\]
Thus, for $\frac1{r} = \frac12 + \frac1{2q},~ \frac1{r_1} = 1-\frac3{2q}$ and $\frac1{r_2} = 1 - \frac2q$,
\[
\begin{split}
	\bigl\| |x|^2 \nabla v \cdot \nabla^\bot \psi \bigr\|_{L^{q'} (\mathbb{R}^2)}
	&\le
	C \bigl\| \theta \bigr\|_{L^1 (\mathbb{R}^2)} \bigl\| \nabla v \bigr\|_{L^{q'} (\mathbb{R}^2)}
	+
	C \bigl\| |x| \theta \bigr\|_{L^{r} (\mathbb{R}^2)} \bigl\| \nabla v \bigr\|_{L^{r_1} (\mathbb{R}^2)}\\
	&+
	C \bigl\| |x|^2 \theta \bigr\|_{L^q (\mathbb{R}^2)} \bigl\| \nabla v \bigr\|_{L^{r_2} (\mathbb{R}^2)}.
\end{split}
\]
Employing \eqref{decay-drv}, we obtain $\| \nabla v \|_{L^{q'} (\mathbb{R}^2)} \le C (1+t)^{-\frac2\alpha-\frac2{\alpha q}},~ \| \nabla v \|_{L^{r_1} (\mathbb{R}^2)} \le C (1+t)^{-\frac2{\alpha} -\frac3{\alpha q}}$ and $\| \nabla v \|_{L^{r_2} (\mathbb{R}^2)} \le C (1+t)^{-\frac2\alpha-\frac4{\alpha q}}$.
Therefore
\begin{equation}\label{bs2}
\begin{split}
	\biggl| \int_0^t \int_{\mathbb{R}^2}
		|x|^4 \theta \nabla v \cdot \nabla^\bot \psi
	dxds \biggr|
	\le
	C.
\end{split}
\end{equation}
Since \eqref{str} with $v$ instead of $\theta$ holds,
the last term of \eqref{bs} is estimated by
\[
\begin{split}
	&\biggl|
		\int_{\mathbb{R}^2}
			|x|^2 v \bigl[ |x|^2, (-\Delta)^{\alpha/2} \bigr] v
		dx
	\biggr|
	\le
	\bigl\| |x|^2 v \bigr\|_{L^{2_*} (\mathbb{R}^2)}
	\bigl\| \bigl[ |x|^2, (-\Delta)^{\alpha/2} \bigr] v \bigr\|_{L^{2_*'} (\mathbb{R}^2)}\\
	&\le
	C \bigl\| (-\Delta)^{\alpha/4} (|x|^2 v) \bigr\|_{L^2 (\mathbb{R}^2)}
	\sum_{|\beta| + |\gamma| = 2,~ |\beta| \ge 1}
	\bigl\| (-\Delta)^{\frac\alpha{2} - \frac{|\beta|}2} (x^\gamma v) \bigr\|_{L^{2_*'} (\mathbb{R}^2)},
\end{split}
\]
where $\frac1{2_*} = \frac12 - \frac\alpha{4}$ and then $\frac1{2_*'} = \frac12 + \frac\alpha{4}$.
Moreover there exist constants $b_{\gamma_1,\gamma_2}$ and $b_0$ such that
\[
\begin{split}
	\mathcal{F} [x^\gamma v] (t)
	&=
	\sum_{|\gamma_1|+|\gamma_2| = |\gamma|} b_{\gamma_1,\gamma_2} \int_0^t i\xi (i\nabla)^{\gamma_1} (e^{-(t-s)|\xi|^\alpha})
	\cdot (i\nabla)^{\gamma_2} \mathcal{F} [\theta\nabla^\bot\psi] (s,\xi) ds\\
	&+
	b_0 \int_0^t e^{-(t-s)|\xi|^\alpha} \mathcal{F} [\theta\nabla^\bot\psi] (s,\xi) ds.
\end{split}
\]
Here $b_0 = 0$ when $|\gamma| = 0$.
Thus
\begin{equation}\label{ir}
\begin{split}
	(-\Delta)^{\frac{\alpha-|\beta|}{2}} (x^\gamma v)(t)
	&=
	\sum_{|\gamma_1|+|\gamma_2| = |\gamma|} b_{\gamma_1,\gamma_2}
	\int_0^t
		\nabla (-\Delta)^{\frac{\alpha-|\beta|}{2}} (x^{\gamma_1} G_\alpha) (t-s)* (x^{\gamma_2} \theta\nabla^\bot\psi) (s)
	ds\\
	&+
	b_0 \int_0^t (-\Delta)^{\frac{\alpha-|\beta|}{2}} G_\alpha (t-s)* (\theta\nabla^\bot\psi) (s) ds.
\end{split}
\end{equation}
For some $r_1$ and $r_2$ with $\frac2{3+\alpha-|\beta|-|\gamma_1|} < r_1 < \frac2{3-|\beta|-|\gamma_1|}$ and $1+\frac1{2_*'} = \frac1{r_1} + \frac1{r_2}$, we see  from \eqref{decayG} and \eqref{fracint} that $\nabla (-\Delta)^{\frac{\alpha-|\beta|}{2}} (x^{\gamma_1} G_\alpha) \in L^{2_*'} (\mathbb{R}^2) \cap L^{r_1} (\mathbb{R}^2)$, and obtain by Hausdorff-Young's inequality and Proposition \ref{prop-wt} that
\[
\begin{split}
	&\biggl\| 
		\int_0^t \nabla (-\Delta)^{\frac{\alpha-|\beta|}{2}} (x^{\gamma_1} G_\alpha) (t-s)* (x^{\gamma_2} \theta\nabla^\bot\psi) (s) ds
	\biggr\|_{L^{2_*'} (\mathbb{R}^2)}\\
	&\le
	C \int_0^{t/2}
		\bigl\| \nabla (-\Delta)^{\frac{\alpha-|\beta|}{2}} (x^{\gamma_1} G_\alpha) (t-s) \bigr\|_{L^{2_*'} (\mathbb{R}^2)}
		\bigl\| x^{\gamma_2} (\theta\nabla^\bot\psi) (s) \bigr\|_{L^1 (\mathbb{R}^2)}
	ds\\
	&+
	C \int_{t/2}^t
		\bigl\| \nabla (-\Delta)^{\frac{\alpha-|\beta|}{2}} (x^{\gamma_1} G_\alpha) (t-s) \bigr\|_{L^{r_1} (\mathbb{R}^2)}
		\bigl\| x^{\gamma_2} (\theta\nabla^\bot\psi) (s) \bigr\|_{L^{r_2} (\mathbb{R}^2)}
	ds\\
	&\le
	C \int_0^{t/2}
		(t-s)^{ - \frac2\alpha - \frac12 + \frac{|\beta|}\alpha + \frac{|\gamma_1|}\alpha}
		(1+s)^{-\frac2\alpha + \frac{|\gamma_2|}\alpha}
	ds\\
	&+
	C \int_{t/2}^t
		(t-s)^{-\frac2\alpha (1-\frac1{r_1}) - \frac1\alpha - 1 + \frac{|\beta|}\alpha + \frac{|\gamma_1|}\alpha}
		(1+s)^{-\frac2\alpha (1-\frac1{r_2}) - \frac2\alpha + \frac{|\gamma_2|}\alpha}
	ds.
\end{split}
\]
Here the decay of $\nabla (-\Delta)^{\frac{\alpha-|\beta|}{2}} (x^{\gamma_1} G_\alpha)$ in time has been provided from \eqref{scG}.
The singularity at $t = 0$ is avoided by the similar calculus as above.
Thus
\[
	\biggl\| 
		\int_0^t \nabla (-\Delta)^{\frac{\alpha-|\beta|}{2}} (x^{\gamma_1} G_\alpha) (t-s)* (x^{\gamma_2} \theta\nabla^\bot\psi) (s) ds
	\biggr\|_{L^{2_*'} (\mathbb{R}^2)}
	\le
	C (1+t)^{-1/2} L_\alpha (t),
\]
where
\[
	L_\alpha (t)
	= \left\{
	\begin{array}{lr}
		\log (2+t),	&	\alpha = 1,\\
		1,		&	0 < \alpha < 1.
	\end{array}
	\right.
\]
%
We estimate the second term of \eqref{ir} when $|\gamma| = 1$, i.e., $|\beta|=1$, then the similar argument as above says that
\[
\begin{split}
	&\biggl\|
		\int_0^t (-\Delta)^{\frac{\alpha-|\beta|}{2}} G_\alpha (t-s)* (\theta \nabla^\bot \psi) (s) ds
	\biggr\|_{L^{2_*'} (\mathbb{R}^2)}
	\le
	C (1+t)^{-1/2}.
\end{split}
\]
Therefore
\[
\begin{split}
	\bigl\| \bigl[ |x|^2, (-\Delta)^{\alpha/2} \bigr] v \bigr\|_{L^{2_*'} (\mathbb{R}^2)}
	&\le
	C \sum_{|\beta|+|\gamma|=2,~ |\beta| \ge 1}
	\bigl\| (-\Delta)^{\frac{\alpha-|\beta|}{2}} (x^\gamma v)(t) \bigr\|_{L^{2_*'}  (\mathbb{R}^2)}
	\le
	C (1+t)^{-1/2} L_\alpha (t),
\end{split}
\]
and
\begin{equation}\label{bs3}
\begin{split}
	&\biggl| \int_0^t \int_{\mathbb{R}^2}
		|x|^2 v \Bigl[ |x|^2, (-\Delta)^{\alpha/2} \Bigr] v
	dxds \biggr|
	\le
	C_\delta \log (2+t) L_\alpha (t)^2
	+
	\delta \int_0^t \bigl\| (-\Delta)^{\alpha/4} (|x|^2 v) \bigr\|_{L^2 (\mathbb{R}^2)}^2 ds.
\end{split}
\end{equation}
%
Applying \eqref{bs1}, \eqref{bs2} and \eqref{bs3} into \eqref{bs}, and choosing $\delta$ sufficiently small, we see that
\[
	\bigl\| |x|^2 v \bigr\|_{L^2 (\mathbb{R}^2)}^2
	+
	\int_0^t
		\bigl\| (-\Delta)^{\alpha/4} (|x|^2 v) \bigr\|_{L^2 (\mathbb{R}^2)}^2
	ds
	\le
	C \log (2+t) L_\alpha (t)^2.
\]
A coupling of this estimate and Lemma \ref{wtasymplin} yields the assertion.\hfill$\square$

\end{document}